\documentclass[11pt,english]{article}
\usepackage[T1]{fontenc}
\usepackage[latin9]{inputenc}
\usepackage{amsmath}
\usepackage{amsthm}
\usepackage{amssymb}

\makeatletter

\newcommand{\noun}[1]{\textsc{#1}}

  \theoremstyle{plain}
  \newtheorem*{thm*}{\protect\theoremname}
  \theoremstyle{plain}
  \newtheorem*{cor*}{\protect\corollaryname}
\theoremstyle{plain}
\newtheorem{thm}{\protect\theoremname}
  \theoremstyle{definition}
  \newtheorem{example}[thm]{\protect\examplename}
  \theoremstyle{plain}
  \newtheorem{lem}[thm]{\protect\lemmaname}
  \theoremstyle{plain}
  \newtheorem{prop}[thm]{\protect\propositionname}
  \theoremstyle{plain}
  \newtheorem{cor}[thm]{\protect\corollaryname}
  \theoremstyle{plain}
  \newtheorem{conjecture}[thm]{\protect\conjecturename}
  \theoremstyle{definition}
  \newtheorem*{example*}{\protect\examplename}
\newcommand{\lyxaddress}[1]{
\par {\raggedright #1
\vspace{1.4em}
\noindent\par}
}

\usepackage{mathrsfs}\usepackage{amsthm}\usepackage{amscd}

\usepackage{hyperref}
\usepackage{graphics}
\usepackage{a4wide}
\@ifundefined{definecolor}
 {\usepackage{color}}{}

\newcounter{EQNR}
\renewcommand{\contentsname}{Contents\\
{\footnotesize\normalfont(A table of contents should normally not
be included)}}

\makeatother

\usepackage{babel}
  \providecommand{\conjecturename}{Conjecture}
  \providecommand{\corollaryname}{Corollary}
  \providecommand{\examplename}{Example}
  \providecommand{\lemmaname}{Lemma}
  \providecommand{\propositionname}{Proposition}
  \providecommand{\theoremname}{Theorem}
\providecommand{\theoremname}{Theorem}

\begin{document}

\title{The stars at infinity in several complex variables}

\author{Anders Karlsson\footnote{The author was supported in part by the Swiss NSF grants 200020-200400 and 200021-212864, and the Swedish Research Council grant 104651320.}}

\date{May 28, 2023}
\maketitle
\begin{abstract}
This text reviews certain notions in metric geometry that may have
further applications to problems in complex geometry and holomorphic
dynamics in several variables. The discussion contains a few unrecorded
results and formulates a number of questions related to the asymptotic
geometry and boundary estimates of bounded complex domains, boundary
extensions of biholomorphisms, the dynamics of holomorphic self-maps,
Teichmüller theory, and the existence of constant scalar curvature
metrics on compact Kähler manifolds.
\end{abstract}

\section{From complex to metric}

Pick's striking reformulation of the Schwarz lemma in \cite{Pi16}
started a rich development of metric methods in complex analysis.
The Schwarz-Pick lemma states that every holomorphic self-map of the
unit disk $D$ does not increase distances in the Poincaré metric.
Modifying a definition of Carathéodory, Kobayashi defined in the 1960s
a largest pseudo-distance $k_{Z}$ on each complex space $Z$ so that
every holomorphic map is nonexpansive in these distances. A pseudo-distance
is a metric except that there may exist distinct points $x$ and $y$
such that $d(x,y)=0$. Indeed, $k_{\mathbb{C}}\equiv0$. On the other
hand, $k_{D}$ is the Poincaré metric and one sees that Liouville's
theorem that bounded entire functions are constant immediately follows
from these assertions. This is not a shorter proof of this theorem,
but places it into a different framework that also explains the Picard's
little theorem. Complex space where the Kobabyshi pseudo-distance
is an actual distance are called \emph{Kobayashi hyperbolic} and this
is via Lang's conjectures from the 1970s connected to finiteness of
the number of rational solutions to diophantine equations \cite{La74,La86}.

Metric methods are thus rather old in the subject of several complex
variables, but as seen for example in \cite{BG20,BGZ21,AFGG22,BNT22,Z22,GZ22,L22},
the metric perspective has developed strongly also in recent years.
The present paper tries to outline a few further possible directions
mostly to do with boundaries and intrinsic structures on them. One
classic topic in complex analysis is the question of the extension
of holomorphic maps to the boundary of the domain, a problem for which
also the Bergman metric has been used. In Riemannian geometry, boundary
maps were considered in the 1960s by Mostow for the proofs of his
landmark rigidity theorems. At an early stage his ideas did not attract
much interest from the Lie group community, instead encouragement
came from Ahlfors, the famous complex analyst \cite{Mo96}. Nowadays,
since the influence of Gromov having Mostow's and Margulis' work as
a starting point, boundaries at infinity and extensions of maps between
them are a staple of geometric group theory. In completing the circle
as it were, these metric ideas have since been applied to the above-mentioned
complex analytic question \cite{BB00,CL21}.

In another direction, an important problem in complex geometry is
to understand when a given compact Kähler manifold admits a constant
scalar curvature Kähler metric in the same cohomology class. This
study started in the 1950s by Calabi who for this purpose introduced
a flow on a certain space of metrics on the underlying manifold. In
response to a conjecture of Donaldson, Streets proposed to study a
related weak Calabi flow on the metric completion of Calabi's space
equipped with a certain $L^{2}$-type metric. This flow is nonexpansive
in this metric and therefore metric versions of the Denjoy-Wolff theorem
could be relevant for this set of problems. This is discussed more
in section \ref{sec:The-Calabi-flow} and we refer to \cite{CC02,BDL17,CCh21}
for references on this topic. We thus see an example of an entirely
different use of metric methods, more in the spirit of Teichmüller
theory, for problems in several variable complex geometry. 

The present paper reviews the stars from \cite{K05}. Given a metric
space $X$ with a compactification $\overline{X}$, we associate an
extra structure of the boundary $\partial X:=\overline{X}\setminus X$.
This boundary structure consists of subsets called \emph{stars}, which
are limits of generalized halfspaces, and is an isometry invariant
in case the isometries act naturally on the boundary. Everything is
defined in terms of distances, without any knowledge about the existence
of geodesics. The interest of these notions comes from that:
\begin{itemize}
\item the stars measure the failure of Gromov hyperbolicity (\emph{e.g.
}Proposition \ref{prop:visib}), and this is useful when extending
the theory of Gromov hyperbolic spaces to more general metric spaces
(\emph{e.g. }\cite[section 4]{K05});
\item the stars provide a way of describing the asymptotic geometry in any
given compactification, and boundary estimates can sometimes be translated
into qualitative information (\emph{e.g. }Theorem \ref{thm:visible},
\cite[sections 8,9]{K05}, \cite{DF21});
\item the stars restrict the limit sets of nonexpansive maps (Theorem \ref{thm:WDstar});
\item a contraction lemma dictates the dynamics of isometries (Lemma \ref{lem:contraction-lemma});
\item there are conceivable versions of boundary extensions of isometric
maps (\emph{cf. }Question G).
\end{itemize}
As I have argued in a paper originally entitled \emph{From linear
to metric} \cite{K21b}, the methods discussed here are also useful
in other mathematical and scientific subjects. And while the metric
ideas in the present paper are rather unified and focused, their consequences
belong to a diverse set of topics. Let me illustrate the latter by
highlighting three such results. The final section establishes:
\begin{thm*}
(Theorem \ref{thm:randomvisible}) Let $R_{n}=f_{1}\circ f_{2}\circ...\circ f_{n}$
be the composition of randomly selected holomorphic self-maps of a
bounded domain $X$ in $\mathbb{C}^{N}$. Let $d$ denote the Kobayashi
distance and assume $X$ is a weak visibility domain in the sense
of Bharali-Zimmer. Then a.s. unless
\[
\frac{1}{n}d(x,R_{n}x)\rightarrow0,
\]
as $n\rightarrow\infty$, there is a random point $\xi\in\partial X$
independent of $x$ such that 
\[
R_{n}x\rightarrow\xi,
\]
as $n\rightarrow\infty$. 
\end{thm*}
In section \ref{sec:The-Calabi-flow} the following improvement of
\cite[Corollary 4.2]{X21} is obtained:
\begin{cor*}
(Corollary \ref{cor:Xia}) Let $(Y,\omega)$ be a compact Kähler manifold
that is geodesically unstable. Let $\Phi_{t}$ be the weak Calabi
flow on the associated completed space $\mathcal{E}^{2}$ of Kähler
metrics in the same class. Then $\Phi_{t}(x)$ lies on sublinear distance
to a unique geodesic ray $\gamma$ as $t\rightarrow\infty$.
\end{cor*}
In particular, which is a significantly weaker statement, $\Phi_{t}(x)$
converges as $t\rightarrow\infty$ to the point defined by $\gamma$
in the visual boundary. This type of statement was established and
exploited in \cite{CS14} for the purpose of partially confirming
a conjecture of Tian.

Finally, recall the notion of extremal length initially studied by
Grötzsch, Beurling and Ahlfors. Let $M$ be a closed surface with
complex structure $x$ and $\alpha$ an isotopy class of a simple
closed curve on $M$, and define
\[
\mathrm{Ext}_{x}(\alpha)=\sup_{\rho\in[x]}\frac{\ell_{\rho}(\alpha)}{\mathrm{Area}(\rho)},
\]
where the supremum is taken over all metrics in the conformal class
of $x$. We have from section \ref{sec:Wolff-Denjoy-type-theorems}:
\begin{cor*}
(Corollary \ref{cor:extremallength}) Let $f$ be a holomorphic self-map
of the Teichmüller space of $M$. Then there exists a simple closed
curve $\beta$ such that 
\[
\lim_{n\rightarrow\infty}\mathrm{Ext}_{f^{n}(x)}(\beta)^{1/n}=\lim_{n\rightarrow\infty}\left(\sup_{\alpha}\frac{\mathrm{Ext}_{f^{n}(x)}(\alpha)}{\mathrm{Ext}_{x}(\alpha)}\right)^{1/n},
\]
where the supremum is taken over all simple closed curves on $M$.
\end{cor*}

\paragraph*{Acknowledgements: }

Part of this text was presented at the INdAM workshop in the Palazzone
in Cortona, Tuscany in September of 2021. I heartily thank Filippo
Bracci, Hervé Gaussier, and Andrew Zimmer, for organizing this stimulating
workshop and for their generous invitation. It happened to take place
around the time of a pandemic and just a week short of the 700th anniversary
of the death of Dante Alighieri, the great Florentine poet. Given
this context, upon arrival in Cortona admiring the clear starry night
sky, and in view of the topic of my lecture, however modest, it was
impossible not to associate to Dante since the concept of stars is
such an important one in his \emph{Commedia}. 

\begin{center}  \emph{``E quindi uscimmo a riveder le stelle.''}

\end{center}

\section{Boundaries, halfspaces, and stars at infinity}

\paragraph*{Boundaries.}

Let $(X,d)$ denote a metric space, for example a complex domain with
the Kobayashi metric assumed (Kobayashi) hyperbolic. Let $\overline{X}$
be a compactification of $X$ with which we mean a compact Hausdorff
space $\overline{X}$ and a topological embedding $i:X\rightarrow\overline{X}$
such that $\overline{X}=\overline{i(X)}$. Often we suppress the map
$i$ and consider $X$ simply as a subset of $\overline{X}$. In the
best case scenario the isometries of $X$ extend naturally to homeomorphisms
of $\overline{X}$, in this case we speak of an $Isom(X)$-compactification.
The corresponding \emph{(ideal) boundary} is $\partial X:=\overline{X}\setminus i(X)$.
There is also a weaker notion that only insists that $i$ is continuous
and injective, but that $X$ is not necessarily homeomorphic to $i(X)$,
we will refer to this as a \emph{weak compactification.} 
\begin{example}
In the case of bounded domains in complex vector spaces $\mathbb{C}^{N}$,
we can consider the closure which often is called the \emph{natural}
boundary, here I will call it the \emph{natural extrinsic} boundary.
Complex automorphisms of the domain may not extend to the boundary,
the question of when they do is a classical topic as mentioned in
the introduction. 
\end{example}

\begin{example}
For any metric space there is always an intrinsic method of a weak
$Isom(X)$-compactification, called \emph{the horofunction bordification}
or as I prefer, following Rieffel, \emph{the metric compactification},
which is increasingly seen as being of fundamental importance. Given
a base point $x_{0}$ of the metric space $X$, let 
\[
\Phi:X\rightarrow\mathbb{R\mathrm{^{X}}}
\]
be defined via
\[
x\mapsto h_{x}(\cdot):=d(\cdot,x)-d(x_{0},x).
\]
With the topology of pointwise convergence, this is a continuous injective
map and the closure $\overline{\Phi(X)}$ is compact and Hausdorff.
The elements of $\overline{X}^{h}:=\overline{\Phi(X)}$ are called
\emph{metric functionals}. In the case $X$is a proper geodesic space
this construction gives a compactification in the stricter sense.
If $X$ is a proper metric space the elements of $\partial X$ are
called \emph{horofunctions}, the typical example of such comes from
geodesic rays and are called \emph{Busemann functions. }(in the literature
starting with Rieffel this word is also used for limits along almost
geodesics). 
\end{example}

A major question is to investigate the exact relation between these
two examples: the metric compactification in the Kobayashi metric,
which is the natural intrinsic compactification, and the natural extrinsic
boundary of bounded complex domains. See for example \cite{AFGG22}.
Let me record it as follows:

\noun{Question A: }\emph{Let $X$ be a bounded complex domain with
the Kobayashi metric. How and when can one relate the natural extrinsic
boundary with the boundary from the metric compactification? }

Let $(X,d)$ denote a metric space and $\overline{X}$ a compactification
of $X$. For a subset $W\subset\overline{X}$ we set 
\[
d(x,W):=\inf_{w\in W\cap X}d(x,w).
\]

\paragraph*{Halfspaces. }

Let $x_{0}\in X$. The \emph{(generalized) halfspace }defined by $W\subset\overline{X}$
and real number $C$ is
\[
H(W,C):=H^{x_{0}}(W,C):=\left\{ z:d(z,W)\leq d(z,x_{0})+C\right\} .
\]
We also use the notation $H(W):=H(W,0)$. These notions have the advantage
that they do not refer to geodesics which may or may not exist. 

\paragraph*{Stars at infinity.}

Let $\xi\in\partial X$ and denote by $\mathcal{V}_{\xi}$ the collection
of open neighborhoods of $\xi$ in $\overline{X}$. The \emph{star
of }$\xi$ \emph{based at }$x_{0}$ is
\[
S^{x_{0}}(\xi)=\bigcap_{V\in\mathcal{V}_{\xi}}\overline{H(V)},
\]
where the closure is taken in $\overline{X}$, and the \emph{star
of }$\xi$ is 
\[
S(\xi)=\overline{\bigcup_{C\geq0}\bigcap_{V\in\mathcal{V}_{\xi}}\overline{H(V,C)}}.
\]

It is immediate that $\xi\in S^{x_{0}}(\xi)\subset S(\xi)$. The second
definition removes an a priori dependence on $x_{0}$. In all examples
I can think of, the two definitions actually coincide, which in other
words means that the first definition is independent of the base point
chosen. 

\noun{Question B: }\emph{When is $S(\xi)=S^{x_{0}}(\xi)$?}

The motivation for calling these sets \emph{stars }is that they are
subsets at infinity of the space and they have a tendency to be star-shaped
in appropriate senses or even to coincide with the notion of star
in the theory of simplicial complexes. In addition they are closely
related to visibility properties of the compactification (see Proposition
\ref{prop:visib}) so one could appeal to the light emitting property
of physical stars.

A \emph{face }is a non-empty intersection of stars. The following
notion will be used below, the \emph{dual star of }$\xi$:
\[
S^{\vee}(\xi):=\left\{ \eta\in\partial X\,:\,\xi\in S(\eta)\right\} 
\]

It was observed in \cite{K05} that in all the examples considered
there, $S^{\vee}(\xi)=S(\xi)$, which could be called \emph{star-reflexivity},
and raised the question whether or when it is the case. In an insightful
paper by Jones and Kelsey \cite{JK22} examples of homogeneous graphs,
certain Diestel-Leader graphs, with their metric compactification
were shown \emph{not }to have this property. Understanding this phenomenon
better has some additional interest in view of results like Theorem
\ref{thm:WDstar} or Proposition \ref{prop:horoball} below.

\noun{Question C: }\emph{Which spaces and compactifications are star
reflexive in the sense that }$S^{\vee}(\xi)=S(\xi)$ \emph{for all
$\xi\in\partial X$?}

An obvious case of $S^{\vee}(\xi)=S(\xi)$ is when $S(\xi)=\left\{ \xi\right\} .$
Such points are called \emph{hyperbolic }because for example for any
Gromov hyperbolic space with its standard Gromov boundary every point
is hyperbolic. We call compactifications with this property \emph{hyperbolic.
}Another classical example of a hyperbolic compactification is the
end compactification of any proper metric space. For bounded complex
domains with Kobayashi metric, every natural extrinsic boundary point
that is $C^{2}$-smooth and strictly pseudoconvex is hyperbolic as
P.J. Thomas informed me, see also below and \cite[Corollary 35]{K05}.

\noun{Question D: }\emph{For convex or pseudoconvex domains with the
usual boundary, what are the Kobayashi stars?}

Hilbert's metric on convex domains has a very simple answer provided
in \cite[Proposition 32]{K05}: the stars and faces coincides with
the usual homonymous notions for convex sets. 

Teichmüller spaces of closed surfaces are of great importance and
are biholomorphic to bounded pseudoconvex domains. They have a natural
metric, the Teichmüller metric, which Royden showed coincides with
the Kobayashi metric. The recent paper by Duchin and Fisher \cite{DF21}
makes substantial progress toward \cite[Conjecture 46]{K05} determining
the stars in the Teichmüller metric with the Thurston compactification
which is defined in terms of topology and hyperbolic geometry. It
has been observed for a long time that the complex analytic notions
have a complicated relationship with the concepts from the approach
of hyperbolic geometry, but here there is a hope of a clean tight
connection. 

One can define the \emph{star-distance} $d_{\star}$ as in \cite{K05,DF21}
to be the induced path distance on $\partial X$, in the extended
sense that distances may be infinite, from defining 
\[
d_{\star}(\xi,\eta)=0\;\Longleftrightarrow\:\xi=\eta
\]
and if $\xi\neq\eta$
\[
d_{\star}(\xi,\eta)=1\;\Longleftrightarrow\:\eta\in S(\xi)\textrm{ or }\xi\in S(\eta).
\]

In the case of an $Isom(X)$-compactification the isometries obviously
act by isometry also on $(\partial X,d_{\star})$. It is a trivial
concept in case $X$ is Gromov hyperbolic with its standard boundary.
This is related to Tits incidence geometry at infinity in nonpositive
curvature and conjecturally (\cite{K05,DF21}) the star distance restricted
to the simple closed curves of the Thurston boundary isometric to
the curve complex defined in pure topological terms, see \cite{DF21}
for more details and for a conjectural outline arriving at such a
result.

Let me mention the following useful fact, the \emph{sequence criterion
}for star membership of Duchin-Fisher extending a lemma in \cite{JK22}:
\begin{lem}
\cite{DF21} Let $(X,d)$ be a metric space and $\overline{X}$ a
compactification of $X$. Assume that $\overline{X}$ is first countable.
Then $\eta\in S(\xi)$ if and only if for every neighborhood $U$
of $\eta$ in $\overline{X}$, there are sequences $x_{n}\rightarrow\xi,$
$y_{n}\rightarrow U$ and a constant $C\geq0$ such that 
\[
d(y_{n},x_{n})\leq d(y_{n},x_{0})+C.
\]
In particular, if there is such sequences with $y_{n}\rightarrow\eta$,
then $\eta\in S(\xi)$.
\end{lem}

Isometries, when well-defined as maps of $\overline{X}$, preserve
the star distance. This stands in contrast to an opposite phenomenon
namely that topologically isometries tend to have strong contraction
properties on $\overline{X}$ as expressed by Lemma \ref{lem:contraction-lemma}
below. Especially this is the case when there are many hyperbolic
points but on the other hand it can reduce to no contraction for example
in case of the euclidean spaces with the usual visual boundaries (if
one takes an $\ell^{1}$ metric instead there is some contraction).
The \emph{north-south dynamics} is one of the most important features
in the theory of word hyperbolic group, and states that for any sequence
of group elements $g_{n}$ which converges to $\xi^{+}$ when $n\rightarrow\infty$
and $g_{n}^{-1}$ to $\xi^{-}$ , it holds that for any two neighborhoods
$V^{+}$of $\xi^{+}$ and $V^{-}$ of $\xi^{-}$ eventually everything
outside $U^{-}$ is mapped inside $U^{+}$ by $g_{n}.$ This is generalized
without any hyperbolicity assumption, and to any compactification,
just adding an ``$H$'':
\begin{lem}
\label{lem:contraction-lemma}(The contraction lemma \cite{K05})
Let $(X,d)$ be a metric space and $\overline{X}$ a compactification
of $X$. Let $g_{n}$ be a sequence of isometries such that $g_{n}x_{0}\rightarrow\xi^{+}\in\partial X$
and $g_{n}^{-1}x_{0}\rightarrow\xi^{-}\in\partial X$ as n $n\rightarrow\infty$.
Then for any neighborhoods $V^{+}$ of $\xi^{+}$ and $V^{-}$ of
$\xi^{-}$, there exists $N>0$ such that
\[
g_{n}(X\setminus H(V^{-}))\subset H(V^{+})
\]
for all $n\geq1$.
\end{lem}

In \cite{K05} a refinement in the case $Isom(X)$-compactifications
is also formulated. In words, $g_{n}$ eventually maps everything
outside the star of $\xi^{-}$ into any neighborhood of the star of
$\xi^{+}$. Note that it is allowed that the two boundary points are
the same, such as is the case for iterates of a parabolic isometry
in hyperbolic geometry. The interplay between the invariance of the
star distance and the contractive property of Lemma \ref{lem:contraction-lemma}
can sometimes be used to rule out non-compact automorphism groups,
see \cite[Theorem 4]{K05} for an example. In a more classical direction
it recovers Hopf's theorem on ends which states that any topological
space $X$ that is a regular covering space of a nice compact space
must have either 0, 1, 2 or a continuum, of ends. In particular it
applies to finitely generated groups and the ends of their Cayley
graphs. This generalizes in view of Lemma \ref{lem:contraction-lemma}
to any hyperbolic compactification with stars replacing ends. 

\noun{Question E: }\emph{Are there applications of the tension between
the invariance of the star distance and the contraction lemma also
for groups of biholomorphisms of certain complex domains?}

To exemplify this idea:
\begin{prop}
Let $(X,d)$ be a proper metric space and $\overline{X}$ an $Isom(X)$-compactification
of $X$. Assume that a noncompact group of isometries of $X$ fixes
a finite set $F$ of boundary points. Then $F$ is contained in two
stars.
\end{prop}

\begin{proof}
By properness of $X$, compactness of $\overline{X}$ and the noncompactness
of the isometry group, we can find isometries $g_{n}$ such that $g_{n}x_{0}\rightarrow\xi^{+}\in\partial X$
and $g_{n}^{-1}x_{0}\rightarrow\xi^{-}\in\partial X$ as $n\rightarrow\infty$.
Since $F$ is finite, by passing to a finite index subgroup, which
does not affect the noncompactness, we may assume that the group fixes
the elements of $F$ pointwise. Any point outside the two stars associated
to $\xi^{\pm}$ must be contracted to these stars according to the
contraction lemma. Such a point cannot be fixed, hence $F\subset S(\xi^{+})\cup S(\xi^{-})$.
\end{proof}

\section{Geodesics and boundary estimates\label{sec:Geodesics-and-boundary}}

Let us begin by the following simple observation:
\begin{prop}
\label{prop:geodesic}Let $X$ be a proper metric space with compactification
$\overline{X}$. To any geodesic ray $\gamma$ there is an associated
face of $\partial X$ being the non-empty intersection of all the
stars which contain limit points of $\gamma(t)$ as $t\rightarrow\infty$.
In particular, all limit points of $\gamma(t)$ are contained in this
face at infinity.
\end{prop}

\begin{proof}
Since the space is proper, any geodesic ray only accumulate at the
boundary. Take any two limit points $\gamma(n_{i})\rightarrow\xi$
and $\gamma(k_{j})\rightarrow\eta$. For any $n_{i}>k_{j}$ we have
\[
d(\gamma(n_{i}),\gamma(k_{j}))=n_{i}-k_{j}<d(\gamma(n_{i}),\gamma(0))\leq d(\gamma(n_{i}),x_{0})+d(x_{0},\gamma(0)).
\]
This implies that $\gamma\in S(\eta)$ since $\gamma(n_{i})$ stays
closer to each neighborhood of $\eta$ up to the constant $C=d(x_{0},\gamma(0))$.
since the two limit points were arbitrary, this shows that any limit
point belongs to the star of any other limit point. Thus this intersection
of stars is non-empty and contains all limit points.
\end{proof}
Recall that any geodesic ray defines a Busemann function thus converges
to a boundary point in this compactification. Related to Question
A we have:

\noun{Question F: }\emph{When do Kobayashi geodesic rays converge
to a boundary point in bounded complex domains?}

Partial results follow from works such as \cite{BGZ21,AFGG22} and
earlier papers that identify the natural extrinsic boundary with the
Gromov boundary, since geodesic rays always converge in the latter
boundary. Note that drawing from the analogy with Hilbert's metric
on convex domains (discussed by Vesentini in \cite{V76}) the paper
\cite{FK05} suggests that it could be a more general phenomenon since
it is shown there that Hilbert geodesic rays always converge even
for general convex domains that often are not Gromov hyperbolic.

I think this is related to questions of extending biholomorphisms
$f:X\rightarrow Y$ to the boundaries. Since Kobayashi geodesic rays
are mapped to Kobayashi geodesic rays, and if these, say emanating
from $x_{0}$, are in bijective correspondence with boundary points,
then there is such an extension. In Mostow's work in higher rank he
obtained incidence preserving boundary maps. So even when there are
no well-defined boundary maps one could formulate a vague, more general
question in this direction:

\noun{Question G: }\emph{Are there results of the type that biholomorphims
or proper holomorphic maps induce maps between the face lattices of
the boundaries?}

Some trivial examples and for more discussion of this in the metric
setting, see \cite{K05}. Obviously it will depend on the boundaries,
and one optimistic possibility could be that if one takes the metric
compactification of the domain space, then it would map to the boundary
(or its faces if the boundary is too large) of the range space. Some
interesting results and insightful discussion of related type can
be found in Bracci-Gaussier's papers \cite{BG20,BG22}. In view of
the these papers another question is:

\noun{Question H: }\emph{What are the relations between stars and
the intersection of horoballs with the boundary?}

The following relation is simple for the metric compactification (I
emphasize that we are here primarily discussing horoballs in the above
sense, while there are also the more general notions of Abate's small
and large horoballs defined in 1988 that have been of importance in
complex geometry since then, see \cite{A91}): 
\begin{prop}
\label{prop:horoball}Let $\partial X$ be the metric boundary of
a proper metric space. Let $\mathcal{H}_{\xi}$ be a horoball centered
at $\xi\in\partial X.$ Then 
\[
\overline{\mathcal{H}_{\xi}}\cap\partial X\subset S^{\vee}(\xi).
\]
\end{prop}

\begin{proof}
Let $x_{n}\rightarrow h$ in the metric compactification, which means
that 
\[
h(y)=\lim_{n\rightarrow\infty}d(y,x_{n})-d(x_{0},x_{n}).
\]
Suppose that a sequence $y_{k}$ belongs to the fixed horoball $\mathcal{H}_{\xi}$,
which means that for some $C$ and all $k$
\[
C\geq h(y_{k})=\lim_{n\rightarrow\infty}d(y_{k},x_{n})-d(x_{0},x_{n}).
\]
This implies that for any $C'>C$ and any $k$ there is an $N$ such
that $d(x_{n},y_{k})\leq d(x_{n},x_{0})+C'$ for all $n>N$. From
the definitions we then have that for any limit point $\eta$ of the
sequence $y_{k}$, it holds that $\xi\in S(\eta)$. 
\end{proof}
The\emph{ visibility property }of a compactification has its origin
from Eberlein-O'Neill in nonpositive curvature and has recently entered
into complex analysis in significant ways, see for example \cite{BZ17,BM21,BNT22}
for more discussion. One definition is as follows: for any two boundary
points $\xi$ and $\eta$ there are disjoint closed neighborhoods
$V_{\xi}$ and $V_{\eta}$ and a compact set $K$ such that any geodesic
segment connecting $V_{\xi}$ and $V_{\eta}$ must also meet $K$,
alternatively formulated, there is a bound on the distance from $x_{0}$
to each such geodesics. Real hyperbolic spaces have this property
while Euclidean spaces do not have it, in their standard visual (=metric)
compactifications. In this context the following is immediate:
\begin{prop}
\label{prop:visib}Assume that $X$ is a geodesic space which means
that every two points can be connected by a geodesic segment. Suppose
that for two distinct boundary points, there are disjoint neighborhoods
of them such that all geodesics connecting these neighborhoods have
bounded distance to $x_{0}.$ Then the two stars are disjoint.
\end{prop}

\begin{proof}
The assumption means that the distance between points near $\xi$
and points near $\eta$ is up to a bounded amount the sum of the respective
distance to $x_{0}.$ The conclusion now follows from the definition
of stars: points near one of the boundary point will eventually all
lie outside the halfspaces around the other. 
\end{proof}
If all stars are disjoint, then we must have that $S(\xi)=\left\{ \xi\right\} $
for every $\xi\in\partial X$, and I call such compactifications hyperbolic
as mentioned above. 
\begin{cor}
\label{cor:vis}If a compactification of a geodesic space has the
visibility property, then it is a hyperbolic compactification.
\end{cor}

Since it seems not clear when Kobayashi domains are geodesic spaces,
Bharali and Zimmer, see \cite{BZ17,BM21}, defined a weaker notion
of visibility (see also these papers for a wealth of examples). Let
$X$ be a bounded domain in $\mathbb{C}^{N}$ with its associated
Kobayashi distance $d$. Fix some $\kappa>0$, by \cite[Proposition 4.4]{BZ17}
any two points in $X$ can be joined by a $(1,\kappa)$-almost geodesic
which means a path $\sigma:I\rightarrow X$ such that 
\[
\left|t-s\right|-\kappa\leq d(\sigma(t),\sigma(s))\leq\left|t-s\right|+\kappa
\]
for all $t,s\in I$. Let $\overline{X}$ be the closure $X$ above
referred to as the natural extrinsic compactification of $X$. We
say that $X$ is a \emph{visibility domain }if for any two distinct
boundary points $\xi$ and $\eta$ and neighborhoods $V$ and $W$
in $\overline{X}$ of these two points such that $\overline{V}\cap\overline{W}=\emptyset$,
there exists a compact set $K$ in $X$ such that for any $x\in V\cap X$
and $y\in W\cap X$ and any $(1,\kappa)$-almost geodesic $\sigma$joining
these two points, $\sigma$intersects $K$. 
\begin{thm}
\label{thm:vishyp}Let $X$ be a bounded domain in $\mathbb{C}^{N}$
and $\overline{X}$ its closure. Assume that it is a visibility domain
for the Kobayashi distance. Then
\[
S(\xi)=\left\{ \xi\right\} 
\]
 for every $\xi\in\partial X=\overline{X}\setminus X$.
\end{thm}

The proof is a minor adaptation of Proposition \ref{prop:visib} in
view of \cite[Proposition 4.4]{BZ17}. See also the proof of Theorem
\ref{thm:randomvisible} below.

\noun{Question I: }\emph{Are there in some cases precise relations
between visibility and boundary points being hyperbolic? Are hyperbolic
compactifications a larger class than visibility compactifications?}

This is of interest in the Wolff-Denjoy context discussed below. Here
is a way to get visibility and hyperbolic points from estimates for
the Kobayashi distances, taken from \cite{K05}: 
\begin{thm}
\label{thm:visible}(\cite[Theorem 37]{K05}) Let $X$ be a bounded
$C^{2}$-smooth domain in $\mathbb{C}^{N}$ which is complete in the
Kobayashi metric. Assume that for the infinitesimal Kobayashi metric
$K_{X}(z;v)$ there are some constants $\epsilon>0$ and $c>0$ such
that
\[
K_{X}(z;v)\geq c\frac{\left\Vert v\right\Vert }{\delta(z,\partial X)^{\epsilon}}
\]
for all $z\in X$ and $v\in\mathbb{C}^{N}$, where $\left\Vert \cdot\right\Vert $
and $\delta$ refer to the Euclidean norm and distance respectively.
Then $S(\xi)=\left\{ \xi\right\} $ for every $\xi\in\partial X$
and the compactification has the visibility property.
\end{thm}

The estimate in the assumption of the theorem is established in \cite{Ch92}
for smooth pseudoconvex bounded domains with boundary of finite type
in the sense of D'Angelo. This has subsequently been extended in important
ways, in \cite[Lemma 5]{KT16}, the Goldilocks domain of Bharali-Zimmer
in \cite{BZ17}, and \cite[Theorem 1.5]{BM21}.

\section{Wolff-Denjoy type theorems\label{sec:Wolff-Denjoy-type-theorems}}

An early application of the Schwarz-Pick lemma was found in 1926 seemingly
as a conversation via \emph{Comptes Rendus} of the French Academy
of Sciences between Wolff and Denjoy. It states that any holomorphic
self-map of the unit disk either has a fixed point, or its orbits
converge to a single point in the boundary circle and every horodisk
at that point is an invariant set. Extensions of this has generated
a vast literature, starting with Valiron, Heins, H. Cartan, Hervé,
Vesentini, Abate, Beardon and many others, see \cite{HW21} for references.
Most extensions, assume something like Gromov hyperbolicity or weaker
property (like visibility or strict convexity). The stars will be
used for a weaker conclusion but in a much more general setting. 

I will mention two purely metric versions, one in terms of the metric
compactification and the other one in terms of the stars at infinity
for any given compactfication of interest. 

Let $X$ be a metric space and consider maps $f$ between metric spaces
that are \emph{nonexpansive }in the sense that
\[
d(f(x),f(y))\leq d(x,y)
\]
for all $x,y\in X$. Isometries are important examples and the composition
of nonexpansive maps remains nonexpansive. 

As was remarked in the very beginning Kobayashi provided a functor
from complex spaces and holomorphic maps into psuedo-metric spaces
and nonexpansive maps, thereby constituting a very significant class
of examples. 

One defines the \emph{minimal displacement $d(f)=\inf_{x}d(x,f(x))$
}and the \emph{translation number} $\ensuremath{\tau(f)=\lim_{n\rightarrow\infty}d(x,f^{n}(x))/n}$,
which exists by a well-known subadditivity argument. These numbers
are analogs of the operator norm and spectral radius, respectively,
in particular note that $\tau(f)\leq d(f).$ They have been studied
to some extent in the complex analytic literature, by Arosio, Bracci,
Fiacchi and Zimmer in particular, see \cite{AB16,AFGG22,AFGK22}.

The following, which I think of as a kind of weak spectral theorem
in the metric category \cite{K21b,K18}, can be viewed as a partial
extension of the theorem of Wolff and Denjoy: 
\begin{thm}
\label{thm:metricspectral-1}\emph{(\cite{K01}) }Let $f:X\rightarrow X$
be a nonexpansive map of a metric space $X$. Then there is a metric
functional $h$ such that
\[
h(f^{n}x_{0})\leq-\tau(f)k
\]
for all $n\geq1$, and
\[
\lim_{n\rightarrow\infty}-\frac{1}{n}h(f^{n}x)=\tau(f).
\]
\end{thm}

The theorem implies as very special cases, with geometric input specific
in each case, extensions of the Wolff-Denjoy theorem for holomorphic
maps \cite{K01,K05}, von Neumann's mean ergodic theorem \cite{K18,K21b},
and Thurston's spectral theorem for surface homeomorphisms \cite{K14,H16}.
Moreover, it has been applied in non-linear analysis, see e.g. \cite{LN12}
and gave the classification of isometries of Gromov hyperbolic spaces
even when non-locally compact and non-geodesic. It also provided new
information for isometries of Riemannian manifolds. Its proof has
subsequently been used several times in the setting of Denjoy-Wolff
extensions in several complex variables, see \cite{HW21}. Maybe it
can also be useful for pseudo-holomorphic self-maps, see \cite{BC09}.

To illustrate how a metric generalization of the Wolff-Denjoy theorem
can give back to complex geometry in a different way, we show the
following. Extremal length was defined in the introduction. 
\begin{cor}
\label{cor:extremallength}Let $f$ be a holomorphic self-map of the
Teichmüller space of $M$. Then there exists a simple closed curve
$\beta$ such that 
\[
\lim_{n\rightarrow\infty}\mathrm{Ext}_{f^{n}(x)}(\beta)^{1/n}=\lim_{n\rightarrow\infty}\left(\sup_{\alpha}\frac{\mathrm{Ext}_{f^{n}(x)}(\alpha)}{\mathrm{Ext}_{x}(\alpha)}\right)^{1/n},
\]
where the supremum is taken over all simple closed curves on $M$.
\end{cor}

\begin{proof}
For more information and bibliographic details, see \cite{K14}. Holomorphic
maps do not expand Teichmüller distances since it coincides with the
Kobayashi metric $d$. Kerckhoff showed to following formula (in particular
implying that this expression is symmetric)
\[
d(x,y)=\frac{1}{2}\log\sup_{\alpha}\frac{\mathrm{Ext}_{x}(\alpha)}{\mathrm{Ext}_{y}(\alpha)},
\]
where the supremum is taken of isotopy classes of simple closed curves
on $M$. Denote by $\tau$ the translation length $\tau(f)$ defined
above. Liu and Su showed that the metric compactification coincides
with the Gardiner-Masur compactification, and thanks also to Miyachi
there is a description of the metric functionals. From Theorem \ref{thm:metricspectral-1}
we then have
\[
-\tau n\geq h(f^{n}(x))=\log\sup_{\alpha}\frac{E_{P}(\alpha)}{\mathrm{Ext}_{f^{n}(x)}(\alpha)^{1/2}}-C\geq\log\frac{D}{\mathrm{Ext}_{f^{n}(x)}(\beta)^{1/2}}-C,
\]
for some curve $\beta$ (with $D=E_{P}(\beta)>0$ which must exist)
for all $n>1$. This gives
\[
\mathrm{Ext}_{f^{n}(x)}(\beta)\geq D^{2}e^{2C}e^{2\tau n}.
\]
The other inequality, for any $\epsilon>0$ and all sufficiently large
$n$ follows from the supremum in Kerckhoff's formula: 
\[
\mathrm{Ext}_{f^{n}(x)}(\beta)\leq C_{1}e^{2(\tau+\epsilon)n}.
\]
The result now follows.
\end{proof}
Here is a result that applies to any compactification, in particular
to holomorphic self-maps of bounded domains with the standard boundary
as a subset of $\mathbb{C}^{N}$. 
\begin{thm}
\label{thm:WDstar}(\cite[Theorem 11]{K05}) Let $f:X\rightarrow X$
be a nonexpansive map of a proper metric space. Assume that $\overline{X}$
is a sequentially compact compactification of $X$. Then either the
orbit is bounded or there is a boundary point $\xi\in\partial X$
such that for any $x\in X$, every limit point of $f^{n}(x)$ as $n\rightarrow\infty$
in $\partial X$ is contained in $S^{\vee}(\xi).$ 
\end{thm}

In the usual settings where one assumes something that implies $S^{\vee}(\xi)=\left\{ \xi\right\} $,
we of course may conclude that the orbits converge to this boundary
point, as in the usual Denjoy-Wolff theorem. In view of \cite{DF21}
a corollary, that no doubt is only a partial result, can be formulated:
\begin{cor}
Let $f$ be a holomorphic self-map of the Teichmüller space of a closed
surface. If the orbit is unbounded and has an accumulation point $\xi$
that is a uniquely ergodic foliation in the Thurston boundary, then
$f^{n}(x)\rightarrow\xi$ as $n\rightarrow\infty$ and any $x$.
\end{cor}

A conceivable strengthening of the theorem could be that the limit
set in the unbounded case has to be contained in a single face, compare
with Proposition \ref{prop:geodesic} for an analogy. In order to
be less vague, while more risky, let us formulate the following conjecture:
\begin{conjecture}
Let $X$ be a bounded domain in $\mathbb{C}^{N}$ equipped with Kobayashi
metric and assume it is a proper metric space. Let $\overline{X}\subset\mathbb{C}^{N}$
be the standard closure. Let $f:X\rightarrow X$ be holomorphic. Then
either the orbits stay away from the boundary or there is a closed
face $F\subset\partial X$ in the above metric sense such that for
any $x\in X$ every accumulation point of $f^{n}(x)$ as $n\rightarrow\infty$
belongs to $F$.
\end{conjecture}

Theorem \ref{thm:WDstar} provides a partial result. Note that one
could also ask the same instead using the notion of face\emph{ }as
the a priori non-metric sense of being an intersection of $\partial X$
with a hyperplane. This relates to \cite{A91,AR14} which imply partial
results on the conjecture in cases that the domain is convex or has
a simple boundary.

Here is another partial result:
\begin{prop}
Let $X$ be a bounded domain in $\mathbb{C}^{N}$ equipped with Kobayashi
metric. Let $\overline{X}\subset\mathbb{C}^{N}$ be the standard closure.
Let $f:X\rightarrow X$ be holomorphic. Assume that $d(z,f^{n}(z))\nearrow\infty$
monotonically for some $z\in X$. Then there is a closed face $F\subset\partial X$
in the above metric sense such that for any $x\in X$ every accumulation
point of $f^{n}(x)$ as $n\rightarrow\infty$ belongs to $F$.
\end{prop}

\begin{proof}
Given two subsequences $f^{n_{i}}(z)\rightarrow\xi\in\partial X$
and $f^{k_{j}}(z)\rightarrow\eta$ as $i,j\rightarrow\infty.$ For
any $n_{i}>k_{j}$ we have 
\[
d(f^{n_{i}}(z),f^{k_{j}}(z))\leq d(z,f^{n_{i}-k_{j}}(z))<d(z,f^{n_{i}}(z))\leq d(f^{n_{i}}(z),x_{0})+d(x_{0},z).
\]
For any neighborhood $V$ of $\eta$ we can find a large enough $j$
so that $f^{k_{j}}(z)\in V$, and the above inequality means that
\[
f^{n_{i}}(z)\in H(V,C)
\]
for all $i$ large enough and where $C=d(x_{0},z)$. Hence $\xi\in S(\eta)$.
Since this was for two arbitrary sequence we must also have $\eta\in S(\xi)$
and can conclude that $F$ being the intersection of all the stars
of all accumulation points contains all accumulations points (even
when changing $z$to $x$ since the respective orbits stay on bounded
distance and this does not influence the stars). 
\end{proof}
The same argument would work if one merely knew that for some $a>0$,
$d(z,f^{an}(z))\nearrow\infty$. This is presumably most often the
case, but it may not be so easy to guarantee.

\section{The Calabi flow\label{sec:The-Calabi-flow}}

Given a Kähler manifold with a fixed Kähler class, a natural question
is to determine whether there exists a canonical choice of Kähler
metric in this class. One potential such choice, generalizing Riemann
surface theory and Kähler-Einstein metrics, is to look for metrics
of constant scalar curvature. 

Let $(Y,J,\omega)$ be a compact connected Kähler manifold and consider
the space $\mathcal{H}$ of smooth Kähler metrics in the cohomology
class $[\omega]$, introduced by Calabi in the 1950s. This space can
be equipped with a Riemannian metric (Mabuchi-Semmes-Donaldson) of
Weil-Petersson or Ebin type with nonpositive sectional curvatures
and such that the metric completion $\mathcal{E}^{2}$ is a CAT(0)-space
admitting a concrete description in terms of plurisubharmonic functions
(due to Darvas). The Calabi flow on $\mathcal{H}$ in the space of
metrics does not expand distances as long as it exists. It is believed
to exists for all times. Moreover it is expected that either the flow
converges to a constant scalar curvature metric or it diverges and
should asymptotically contain some information about the Kähler structure
(made precise in a conjecture of Donaldson in terms of geodesic rays).
Streets suggested to study a weak Calabi flow $\Phi_{t}$, which is
nonexpansive being the gradient flow of a convex function $M$, the
$K$-energy of Mabuchi, and exists for all time and coincides with
the Calabi flow when it exists. I refer to \cite{CC02,BDL17,X21}
for more information and appropriate references. The works of Mayer
and Bacak in pure metric geometry play an important role here, let
me add another good reference \cite{CaL10} on gradient flows of convex
functions on CAT(0)-spaces generally. 

From a flow $\Phi_{t}$ we can define the time-one map $f(x)=\Phi_{1}(x)$,
and we have that $f^{n}(x)=\Phi_{n}(x).$ From what has been said
we have that $f$ is a nonexpansive self-map of a complete CAT(0)-space.

Motivated by this it should be useful to see what can be said about
the iteration of nonexpansive maps in the setting of CAT(0)-spaces
(note by the way that since the Kobayashi metric is a supremum-type
metric it is almost never CAT(0) apart from the well-known exceptional
cases). Some results and arguments for locally compact spaces are
contained in Beardon \cite{Be90}.

For general complete CAT(0)-spaces, in the case that the orbits of
$f$ are bounded, several authors have observed the existence of a
fixed point, notably Kirk in \cite{Ki03}. I gave the following argument
in my doctoral thesis:
\begin{prop}
\label{thm:fixed point}Let $f$ be a nonexpansive map of a complete
CAT(0)-space. If the orbits are bounded, then $f$ has a fixed point.
\end{prop}

\begin{proof}
It follows from CAT(0), in fact the uniform convexity, that bounded
subset possesses a unique circumcenter, which is a point that is the
center of a closed ball of minimal radius containing the set. The
uniform convexity also leads to that the intersection of any nested
sequence of decreasing closed convex sets is nonempty, because the
circumcenters form a Cauchy sequence and the limit is a point in the
intersection.

Given an orbit $x_{n}:=f_{n}^{n}(x)$ we can construct an invariant,
bounded, closed convex set in the following way. For each $k>1$ we
let $C_{k}$ be the intersection of the closed balls with centers
at $x_{n}$ and radius the diameter of the orbit. These are non-empty,
bounded and convex. Moreover $f:C_{k}\rightarrow C_{k+1}$ , so the
the closure of the union of all $C_{k}$ is the desired invariant
set. Such sets are ordered by inclusion and every linearly ordered
chain has a lower bound via the intersection, thus Zorn's lemma provides
a minimal element. The circumcenter of this minimal element must be
a fixed point of $f$.
\end{proof}
CAT(0)-spaces have the so-called \emph{visual bordification }consisting
of equivalence classes of geodesic rays (this is is not related to
the visibility property above). For proper spaces it coincides with
the metric compactification. 
\begin{thm}
\label{thm:cat0}Let $f$ be a nonexpansive map of a complete CAT(0)-space
$X$. If the orbits are bounded, then $f$ has a fixed point. If $d_{f}:=\inf_{x}d(x,f(x))>0$,
then there exists a unique geodesic ray $\gamma$ from $x$ such that
\[
\frac{1}{n}d(f^{n}(x),\gamma(d_{f}n))\rightarrow0,
\]
as $n\rightarrow\infty$, and $b(f(x))\leq b(x)-d_{f}$ for all $x$
where $b$ denotes the Busemann function associated to $\gamma$.
In particular, $f(x)$ converges to the class of $\gamma$ in the
visual bordifcation. If $d_{f}=0$, then there is a metric functional
$h$ such that $h(f(x))\leq h(x)$ for all $x$. 
\end{thm}

\begin{proof}
The bounded orbit case is treated in Proposition \ref{thm:fixed point}.
It is known from \cite{GV12} that for any nonexpansive self-map $f$
of a CAT(0)-space, since it is a star-shaped hemi-metric space, the
minimal displacement $d_{f}$ equals the translation length $\tau_{f}$.
Therefore $\tau_{f}>0$, and a special case of the main result in
\cite{KM99} then shows that there is a unique unit-speed geodesic
ray $\gamma$ from $x$ such that
\[
\frac{1}{n}d(f^{n}(x),\gamma(d_{f}n))\rightarrow0.
\]
In particular, which is strictly weaker, $f^{n}(x)$ converges to
the visual boundary point $[\gamma]$ for any $x\in X$. Let $b$
be the Busemann function associated to the geodesic ray emanating
from $x_{0}$ representing the class of $\gamma$. This is the unique
Busemann function such that (\emph{cf. }\cite{KL06}) 
\[
-\frac{1}{n}b(f^{n}x)\rightarrow d_{f}.
\]
 Theorem \ref{thm:metricspectral-1} now implies moreover that $b(f(x))\leq b(x)-d_{f}$
which holds for all $x$ in view of \cite{GV12}. And in the remaining
case that $d_{f}=0$ and the infimum is not attained, Theorem \ref{thm:metricspectral-1}
gives a metric functional $h$ such that $h(f^{n}(x_{0}))\leq0$ for
all $n>0.$ When the space is CAT(0) this can be improved, see \cite{K01}
or \cite{GV12}, to give the remaining assertion for any $x$.
\end{proof}
Anyone from complex dynamics, or attentive reader of the previous
sections, will notice that this too is a Wolff-Denjoy type theorem
in a purely metric setting, but now with implications for complex
geometry. It is known, see \cite{BDL17}, that either the trajectories
of the weak Calabi flow converges to a constant scalar curvature metric,
or they diverge, $d(x,\Phi_{t}(x))\rightarrow\infty$ for any $x$.
In the latter case on would like to know some directional behavior,
for example the notion of weakly asymptotic geodesic ray introduced
by Darvas-He and studied in \cite{BDL17}, which is a notion much
weaker (for example not necessarily unique) than the one in Theorem
\ref{thm:cat0}. In a special case there is more precise asymptotic
convergence to a geodesic ray known \cite[Theorem 6.3]{CS14} used
to prove the main result concerning uniqueness properties of constant
scalar curvature metrics in that same paper. We will compare the above
with Xia's paper \cite{X21}. Having as a starting point an inequality
and conjecture of Donaldson in \cite{Do05}, Xia proved an analog
of the conjecture when enlarging the space from $\mathcal{H}$ to
$\mathcal{E}^{2}$, namely:
\[
C:=\inf_{x\in\mathcal{E}^{2}}\left|(\partial M)(x)\right|=\max\frac{-\mathbf{M\mathrm{(\ell)}}}{\left\Vert \ell\right\Vert }
\]
where the maximum on the right is taken over boundary points / geodesic
rays in $\mathcal{E}^{2}$. The expression $\left|(\partial M)(x)\right|$
is the local upper gradient of the Mabuchi energy. If its infimum
is strictly positive $(Y,\omega)$ is called \emph{geodesically unstable
}(and in particular admitting no constant scalar curvature metric
in its class). Recent results on geodesic stability and the existence
of constant scalar curvature Kähler metrics can be found in \cite{CCh21}.

Let me formulate the following that improves and reproves parts of
\cite[Corollary 1.2, Corollary 4.2]{X21}:
\begin{cor}
\label{cor:Xia}Let $(Y,\omega)$ be a compact connected Kähler manifold
that is geodesically unstable. Let $\Phi_{t}(x)$ be the weak Calabi
flow on the associated space $\mathcal{E}^{2}$ starting from $x$.
Then there exists a unique geodesic ray $\gamma$ from $x$ on sublinear
distance to $\Phi_{t}(x_{0})$, that is 
\[
\frac{1}{t}d(\Phi_{t}(x),\gamma(Cn))\rightarrow0,
\]
where $C$is defined above. In particular $\Phi_{t}(x)\rightarrow[\gamma]\in\partial\mathcal{E}^{2}$
as $t\rightarrow\infty$, for any $x$.
\end{cor}

\begin{proof}
The geodesic unstability asserts that the infimum of the gradient
is strictly positive, $C>0$, which implies that the escape rate of
the weak Calabi flow is linear, see \cite{CaL10}. The corollary is
now a consequence of Theorem \ref{thm:cat0}. 
\end{proof}
There are many studies on the geodesics in spaces like $\mathcal{H}$
and $\mathcal{E}^{p}$, see \cite{B22,AhC22} for two recent papers
in the complex setting and references therein. It might therefore
be of interest that the argument I suggest here (from \cite{KM99})
constructs the geodesic from the flow in a way that does not use any
compactness argument. From my point of view, related to Theorem \ref{thm:metricspectral-1}
above, it should also be fruitful to study the dual notion to geodesics,
the metric functionals of $\mathcal{E}^{1}$ and $\mathcal{E}^{2}$.

Apart from the constant scalar curvature problem, automorphisms of
the underlying Kähler manifold act by isometry on the Calabi space
and also fall under Theorem \ref{thm:cat0}:
\begin{cor}
Let $f$ be a complex automorphism of a compact Kähler manifold. If
\[
d:=\inf_{x}d(x,f(x))>0
\]
 for the action on $\mathcal{E}^{2}$, then there is a unique geodesic
ray $\gamma$ from $x$ such that
\[
\frac{1}{n}d(f^{n}(x),\gamma(dn))\rightarrow0,
\]
 and $f$ fixes the corresponding boundary point in the visual bordification
of $M$. 
\end{cor}

\begin{proof}
For CAT(0)-spaces any boundary limit point of the orbit is fixed by
$f$ as is well known, see for example \cite{K22}, since two sequences
of bounded distance from each other converge to the same equivalence
class of geodesic ray when they converge. 
\end{proof}
Even in the case $d=0$, it holds that $f$ fixes a metric functional
of $\mathcal{E}^{p}$ (\cite{K22}). A condition like $\inf_{x}d(x,f(x))=0$
where $f$ is a diffeomorphism of an underlying compact manifold acting
instead by isometry on a space of metrics was proposed by D'Ambra
and Gromov in \cite{DG91} as \emph{(quasi-) unipotency} of $f$.

\noun{Question J: }\emph{How to describe or understand the metric
functionals for the Calabi-Mabuchi space and what do the relevant
results discussed in this paper imply concretely for the Calabi flow,
Donaldson's conjecture, and automorphisms of the Kähler manifold?}

A final remark is that in the standard visual bordification, the stars
are identified in \cite[Proposition 25]{K05} as:
\[
S(\xi)=S^{x_{0}}(\xi)=\left\{ \eta:\angle(\eta,\xi)\leq\pi/2\right\} =S^{\vee}(\xi),
\]
where $\angle$ denotes the Tits angle between two geodesic rays and
is symmetric in its arguments.
\begin{example*}
The stars for Euclidean spaces are half-spheres (the Tits angle coincides
with the usual notion of angle) and for hyperbolic spaces they are
points (since any two boundary points can be joined by a geodesic
line giving the angle $\pi$).
\end{example*}

\section{Random iteration}

The above discussion has involved iterations $f^{n}$ of a holomorphic
self-map $f$. In some contexts one meets the generalization to the
composition of several different holomorphic maps. When looking at
the asymptotics one can compose them in two ways, forward or backward.
The latter, i.e.
\[
R_{n}=f_{1}\circ f_{2}\circ...\circ f_{n}
\]
behaves best when studying individual orbits and appears for example
in the theory of continued fractions. Indeed, a continued fraction
expansion of a number is exactly such an expression where $f_{i}(z)=a_{i}/(b_{i}+z$)
are certain Möbius maps, and letting $n\rightarrow\infty$. Other
examples considered by Ramanujan and Polya-Szegö are infinite radicals
($f_{i}(z)=\sqrt{a_{i}z+b_{i}}$) and iterated exponentials ($f_{i}(z)=a_{i}^{z}$),
see \cite{Lo99} for more details. In these contexts one considers
the limit of the corresponding $R_{n}(0)$ as $n\rightarrow\infty$.
There is also a connection to Nevanlinna-Pick interpolation.

Some papers on this topic, see for example \cite{BCMN04,AC22,JS22},
take arbitrary sequences of maps (like in iterated function systems
and the theory of fractals) and sometimes call them \emph{random iteration}.
Here we will only discuss $R_{n}$ for actually randomly selected
holomorphic maps $f_{i}$. 

We formalize the setting as follows, more general than the usual \emph{random}
assumption of independently, identically distributed selected maps.
Let $(T,\Omega,\mu)$ be an ergodic measure preserving system with
$\mu(\Omega)=1$. Given a measurable map $f:\Omega\rightarrow G$
into a semigroup, we define the following \emph{ergodic cocycle}:
$R(n,\omega)=f(\omega)f(T\omega)...f(T^{n-1}\omega)$ or in probabilistic
notation leaving out the measure space: $R_{n}=f_{1}f_{2}...f_{n}$.
It is \emph{integrable }if 
\[
\int_{\Omega}d(x,f(\omega)x)d\mu<\infty
\]
for some $x\in X.$ Then by a well-known consequence of Kingman's
subadditive ergodic theorem, the limit exist 
\[
\tau:=\inf_{n}\frac{1}{n}\int d(x,R(n,\omega)x)d\mu=\lim_{n\rightarrow\infty}\frac{1}{n}d(x,R(n,\omega)x)
\]
for almost every $\omega$ and by ergodicity $\tau$ is independent
of $\omega.$ 

Recall the notion of visibility domain from section \ref{sec:Geodesics-and-boundary}. 
\begin{thm}
\label{thm:randomvisible}Let $R_{n}=f_{1}f_{2}...f_{n}$ be an integrable
ergodic cocycle of holomorphic self-maps of a bounded domain $X$
in $\mathbb{C}^{N}$ that is a visibility domain with the Kobayashi
distance $d$. Then almost surely it holds that unless
\[
\frac{1}{n}d(x,R_{n}x)\rightarrow0
\]
as $n\rightarrow\infty$, there is a random point $\xi\in\partial X$
such that 
\[
R_{n}x\rightarrow\xi
\]
 $n\rightarrow\infty$, for any $x\in X.$ 
\end{thm}

\begin{proof}
Fix $x\in X$. Assume that for a.e. $\omega$
\[
\frac{1}{n}d(x,R(n,\omega)x)\rightarrow\tau>0.
\]
Take $0<\epsilon<\tau.$ By \cite[Proposition 4.2]{KM99}, for a.e.
$\omega,$ there is a sequence of $n_{i}\rightarrow\infty$ and $K$
such that 
\[
d(x,R(n_{i},\omega)x)-d(x,R(n_{i}-k,T^{k}\omega)x)\geq(\tau-\epsilon)k
\]
for all $K<k<n_{i}$. Note that $d(R(n_{i},\omega)x,R(k,\omega)x)\leq d(x,R(n_{i}-k,T^{k}\omega)x)$
by the nonexpansive property. This implies that 
\[
d(x,R(n_{i},\omega)x)+d(x,R(k,\omega)x)-d(R(n_{i},\omega)x,R(k,\omega)x)\ge(\tau-\epsilon)k
\]
for all $K<k<n_{i}$ (for large $k$ we could even insert a $2$ on
the right hand side). Hence the left hand side tends to infinity as
$k<n_{i}\rightarrow\infty$. By compactness we may assume that $R(n_{i},\omega)x\rightarrow\xi$
some point $\xi=\xi(\omega)\in\partial X$. 

Now suppose that for some subsequence $k_{j}$, $R(k_{j},\omega)x\rightarrow\eta$
for some other boundary point $\eta$. Fix two disjoint closed neighborhoods
$V$ and $W$ of $\xi$ and $\eta$ respectively. Consider all large
enough $j$ such that $R(k_{j},\omega)x\in W$ and take a corresponding
$i_{j}$ such that $k_{j}<n_{i_{j}}$ and $R(n_{i_{j}},\omega)x\in V$,
and take an almost geodesic $\sigma_{j}$ joining these two orbit
points. By the visibility assumption there is a $C$ independent of
$j$ such that there exists $t_{j}$ for which 
\[
d(x,\sigma_{j}(t_{j}))<C,
\]
for all large $j$. Since $\sigma_{j}$ is an almost geodesic we have
\[
d(R(k_{j},\omega)x,\sigma_{j}(t_{j}))+d(\sigma_{j}(t_{j}),R(n_{i_{j}},\omega)x)\leq d(R(n_{i_{j}},\omega)x,R(k_{j},\omega)x)+3\kappa.
\]
Therefore by the triangle inequality
\[
d(R(n_{i_{j}},\omega)x,R(k_{j},\omega)x)>d(x,R(n_{i_{j}},\omega)x)+d(x,R(k_{j},\omega)x)-2C-3\kappa.
\]
So for an infinitude of $n_{i}$ and $k$s we have
\[
d(x,R(n_{i},\omega)x)+d(x,R(k,\omega)x)-d(R(n_{i},\omega)x,R(k,\omega)x)<2C+3\kappa,
\]
 but this contradicts the previous estimate. The conclusion follows.
\end{proof}
Without the visibility assumption, a more general result (in view
of Theorem \ref{thm:vishyp}) holds corresponding to Theorem \ref{thm:WDstar}:
\begin{thm}
\cite[Theorem 18]{K05}Let $(X,d)$ be a proper metric space, $\overline{X}$
a compactification, and $R_{n}=R(n,\omega)$ an integrable ergodic
cocycle. Assume that $\tau>0.$ Then for a.e. $\omega$ there is a
boundary point $\xi=\xi(\omega)$ such that 
\[
R_{n}x\rightarrow S^{\vee}(\xi)
\]
as $n\rightarrow\infty.$
\end{thm}

The proof is similar to the previously stated theorem. It is not so
easy in general to determine when the linear rate of escape $\tau$
is $0$ or strictly positive. The difference in escape rate can be
exemplified by Brownian motion in euclidean spaces ($\tau=0$ and
no directional convergence) and hyperbolic spaces ($\tau>0$ and asymptotic
convergence). Another result that ultimately might prove to be yet
more general, is the following random version of Theorem \ref{thm:metricspectral-1}
above:
\begin{thm}
\label{thm:ergodic-1-1}\emph{(Ergodic theorem for noncommuting random
operations, \cite{KL06,GK20}) }Let $R_{n}=R(n,\omega)$ be an integrable
ergodic cocycle of nonexpansive maps of a metric space $(X,d)$ assuming
everything is measurable. Then there exists a.s. a metric functional
$h=h^{\omega}$ of $X$ such that
\[
\lim_{n\rightarrow\infty}-\frac{1}{n}h(R_{n}x)=\tau.
\]
\end{thm}

The proof in the isometry case uses the extension of the maps to the
metric compactification, while the general case instead uses intricate
subadditive ergodic theory. How to deduce results like Theorem \ref{thm:randomvisible}
from this latter result is explained in the proof of \cite[Corollary 5.2]{GK20}.
This reference also contains a result on the behavior of the Ahlfors-Beurling
extremal length under the random iteration of holomorphic maps of
Teichmüller spaces.
\begin{example*}
Let $X$ be the Teichmüller space associated to a higher genus closed
orientable surface. Let $R_{n}$ be a non-degenerate random walk on
the group of its complex automorphisms (which by a theorem of Royden
coincides with the mapping class group in topology). The Teichmüller
metric coincides with the Kobayashi metric (again by Royden). The
metric compactification coincides with the Gardiner-Masur compactification
of Teichmüller space \cite{LS14}. The group in question acts properly
on this space that has at most exponential growth and is a non-amenable
group, it then follows from a theorem by Guivarch that then the escape
rate is $\tau>0$. Therefore every random walk converges to a dual
star at infinity of Teichmüller space. See \cite{K14} for more details
and references. Note that \cite{DF21} investigated the stars in the
Thurston compactification, while for this other complex analytic boundary
there has so far not appeared any study of its stars.
\end{example*}

\lyxaddress{Section de mathématiques, Université de Genève, Case postale 64,
1211 Genève, Switzerland; Mathematics department, Uppsala University,
Box 256, 751 05 Uppsala, Sweden.}
\end{document}